\documentclass[11pt]{amsart}
\usepackage{amsfonts,amssymb,amscd,amsmath,enumerate,verbatim}
\usepackage{amsmath,amsthm,amssymb}
\usepackage{ulem,ascmac} 
\usepackage[dvips]{graphicx}
\usepackage{fancybox}

\theoremstyle{plain}
\newtheorem{thm}{Theorem}[section]
\newtheorem{prop}[thm]{Proposition}
\newtheorem{cor}[thm]{Corollary}
\newtheorem{lem}[thm]{Lemma}

\newtheorem{quest}[thm]{Question}
\theoremstyle{definition}
\newtheorem{defn}[thm]{Definition}

\newtheorem{rem}[thm]{Remark}

\newcommand{\conv}{{\mathrm{conv}}}

\newcommand{\heit}{{\mathrm{ht}}}

\newcommand{\supp}{{\mathrm{supp}}}

\newcommand{\R}{{\mathbb{R}}}
\newcommand{\Z}{{\mathbb{Z}}}

\newcommand{\rleft}{\mathopen{}\mathclose\bgroup\left}
\newcommand{\rright}{\aftergroup\egroup\right}

\newcommand{\eb}{{\bf e}}

\def\eb{{\bold e}}

  \makeatletter
    
    \@addtoreset{equation}{section}
  \makeatother

\title{Universal conditions on $h^*$-vectors of lattice simplices}

\author[A.\,Higashitani]{Akihiro Higashitani}
\address[A.\,Higashitani]{Department of Mathematics\\Kyoto Sangyo University\\ 603-8555, Kyoto\\Japan}
\curraddr{}
\email{ahigashi@cc.kyoto-su.ac.jp}
\thanks{The author is partially supported by JSPS Grant-in-Aid for Young Scientists (B) $\sharp$17K14177. }


\begin{document}

\begin{abstract}
In this paper, we will prove that given a lattice simplex with its $h^*$-polynomial $\sum_{i \geq 0}h_i^*t^i$, 
if $h_{k+1}^*=\cdots=h_{2k}^*=0$ holds, then there exists a lattice simplex of degree $k$ 
whose $h^*$-polynomial coincides with $\sum_{i=0}^k h_i^*t^i$. 
Moreover, we will present the examples showing that the condition $h_{k+1}^*=h_{k+2}^*=\cdots=h_{2k-1}^*=0$ is necessary. 
\end{abstract}

\maketitle

\section{Introduction}\label{sec:intro}


A \textit{lattice} polytope is a convex polytope all of its vertices belong to the standard lattice $\Z^d$. 
Let $P \subset \R^d$ be a lattice polytope $P$ of dimension $d$. Given $n \in \Z_{>0}$, 
we consider the number of lattice points $|nP \cap \Z^d|$ contained in the $n$-th dilation of $P$. 
Ehrhart~\cite{Ehr62} proved that $|nP \cap \Z^d|$ can be expressed by a polynomial in $n$ of degree $d$ for each $n$, denoted by $E_P(n)$. 
This polynomial $E_P(n)$ is called the \textit{Ehrhart polynomial} of $P$. 
Moreover, the generating function of $|nP \cap \Z^d|$ becomes the rational function which is of the form 
$$1+\sum_{n = 1}^\infty |nP \cap \Z^d| t^n =\frac{\sum_{i = 0}^dh_i^*t^i}{(1-t)^{d+1}},$$ 
where each $h_i^*$ is an integer. The sequence of integers $h^*(P)=(h_0^*,h_1^*,\ldots,h_d^*)$ is called the \textit{$h^*$-vector} (or \textit{$\delta$-vector}) of $P$, 
and the polynomial $h_P^*(t)=h_0^*+h_1^*t+\cdots+h_d^*t^d$ is called the \textit{$h^*$-polynomial} (or \textit{$\delta$-polynomial}) of $P$. 
We call the degree of the $h^*$-polynomial of $P$ the \textit{degree} of $P$, denoted by $\deg(P)$. 
Note that the degree of a lattice polytope is at most the dimension of polytope. This implies that 
lattice polytopes having degree at most $d$ are a kind of generalization of lattice polytopes of dimension $d$. 

On the $h^*$-vector $h^*(P)=(h_0^*,\ldots,h_d^*)$ of a lattice polytope $P$ of dimension $d$, the following facts are well know. 
\begin{itemize}
\item We have $\displaystyle E_P(n)=\sum_{i=0}^d h_i^*\binom{n+d-i}{d}$. 
\item Some of $h_i^*$ has a simple description in terms of $P$, e.g., $h_0^*=1$, $h_1^*=|P \cap \Z^d| - (d+1)$ and 
$h_i^*=| (d+1-i)P^\circ \cap \Z^d|$ for each $\deg(P) \leq i \leq d$, where $P^\circ$ denotes the relative interior of $P$. 
In particular, $h_d^*=|P^\circ \cap \Z^d|$. Hence, we have \begin{align}\label{ineq_1d}h_1^* \geq h_d^*.\end{align} 
\item The leading coefficient of $E_P(n)$ is equal to $\sum_{i=0}^d h_i^*/d!$ and this coincides with the relative volume of $P$ (\cite[Corollaries 3.20, 3.21]{BR15}). 
The sum $\sum_{i=0}^d h_i^*$ of $h^*$-vector is called the \textit{normalized volume} of $P$. 
\item Each $h_i^*$ is nonnegative (\cite{Sta80}). 
\item If $h_d^* >0$, then we have $h_i^* \geq h_1^*$ for each $1 \leq i \leq d-1$ (\cite{Hib94}). 
\end{itemize}
For more details on Ehrhart polynomials and $h^*$-vectors, consult e.g., \cite{BR15}.

\bigskip


One of the most important problems in Ehrhart theory is the characterization of the polynomials that are the Ehrhart polynomials of some lattice polytopes. 
Since the Ehrhart polynomial and the $h^*$-vector are equivalent, i.e., we can know another one once we know the one, 
we can rephrase this problem as follows: 
\begin{quest}[{cf. \cite[Question 1.1]{BH18}}]
Characterize the sequences $(h_0^*,h_1^*,\ldots,h_d^*)$ of nonnegative integers that are the $h^*$-vectors of some lattice polytope of dimension $d$. 
\end{quest}
It is easy to see that $(1,a)$ is always the $h^*$-vector of some lattice polytope of dimension $1$ for any nonnnegative integer $a$. 
The case of $d=2$ is highly non-trivial, but the following is known: 
\begin{thm}[{\cite{Sco76}}]\label{thm:Scott}
Let $P$ be a lattice polytope $P$ of dimension $2$ with its $h^*$-vector $h^*(P)=(1,h_1^*,h_2^*)$. Then one of the following conditions is satisfied: 
\begin{enumerate}
\item \label{scott:1} $h^*_2 = 0$;
\item \label{scott:2} $h_2^* \leq h^*_1 \leq 3h^*_2 + 3$;
\item \label{scott:3} $h^*_1 = 7$ and $h^*_2=1$.
\end{enumerate}
\end{thm}
We can see that the converse of Theorem \ref{thm:Scott} is also true, i.e., 
we can construct a lattice polytope of dimension $2$ whose $h^*$-vector is equal to $(1,h_1^*,h_2^*)$ for each $(h_1^*,h_2^*)$ 
which satisfies one of the conditions \eqref{scott:1}, \eqref{scott:2} and \eqref{scott:3} in Theorem \ref{thm:Scott}. 

The case $d \geq 3$ is widely open, while some necessary inequalities have been recently conjectured in \cite[Conjecture 8.7]{Balletti} in the case of $d=3$.

As a generalization of Theorem \ref{thm:Scott}, the following is also known: 
\begin{thm}[{\cite[Theorem~2]{Tre10}}]\label{thm:genScott}
Let $P$ be a lattice polytope of degree at most two with its $h^*$-polynomial $h^*_P(t) = 1 + h^*_1 t + h^*_2 t^2$. 
Then one of the following conditions is satisfied: 
\begin{enumerate}
\item \label{genscott:1} $h^*_2 = 0$;
\item \label{genscott:2} $h^*_1 \leq 3h^*_2 + 3$;
\item \label{genscott:3} $h^*_1 = 7$ and $h^*_2=1$. 
\end{enumerate}
\end{thm}
It is proved in \cite[Proposition~1.10]{HT09} that the converse of Theorem \ref{thm:genScott} is also true. 
Note that the inequality $h_1^* \geq h_2^*$ of Theorem~\ref{thm:Scott} comes from \eqref{ineq_1d}, 
while it does not appear in Theorem~\ref{thm:genScott} since the dimension may be greater than two. 

Moreover, as a further generalization of Theorem~\ref{thm:genScott}, the following is proved: 
\begin{thm}[{\cite[Theorem~1.4]{BH18}}]\label{thm:BH}
Let $P$ be a lattice polytope with its $h^*$-polynomial $h_P^*(t) = \sum_{i \geq 0}h_i^*t^i$. Assume $h^*_3=0$. 
Then one of the following conditions is satisfied: 
\begin{enumerate}
\item \label{main:1} $h^*_2 = 0$;
\item \label{main:2} $h^*_1 \leq 3h^*_2 + 3$;
\item \label{main:3} $h^*_1 = 7$ and $h^*_2=1$.
\end{enumerate}
\end{thm}
In \cite{BH18}, the inequalities \eqref{main:1}, \eqref{main:2} and \eqref{main:3} are called \textit{universal} 
because those are valid independently of both the dimension and the degree of lattice polytopes. 
A proof of Theorem \ref{thm:BH} is based on the idea of \textit{spanning polytopes} developed in \cite{HKN16}. 

The motivation to organize this paper is to find a new kind of universal inequalities or \textit{universal conditions}.


\begin{defn}
We say that a sequence $(a_0,a_1,\ldots,a_k)$ of nonnegative integers satisfies \textit{degree $k$ condition} 
if there exists a lattice polytope $P$ of degree at most $k$ such that $h_P^*(t)=\sum_{i=0}^k a_it^i$. 
\end{defn}

For any lattice polytope $P$ with its $h^*$-vector $(h_0^*,h_1^*,\ldots,h_d^*)$, 
we see that $(h_0^*,h_1^*)$ satisfies degree $1$ conditions. 
Moreover, Theorem \ref{thm:BH} says that the assumption ``$h_3^*=0$'' implies degree $2$ condition of $(h_0^*,h_1^*,h_2^*)$ 
of any $h^*$-vector $(h_0^*,h_1^*,\ldots,h_d^*)$ of lattice polytopes. 
Note that the assumption $h_3^* =0$ is necessary for degree $2$ condition. See \cite[Example 1.5]{BH18}. 

Hence, it is natural to think of the assumptions which imply degree $k$ condition for $(h_0^*,h_1^*,\ldots,h_k^*)$ 
of any $h^*$-vector $(h_0^*,h_1^*,\ldots,h_d^*)$ of lattice polytopes in the case $k \geq 3$. The main result of this paper is the following: 
\begin{thm}\label{mainthm}
Let $\Delta$ be a lattice simplex with its $h^*$-polynomial $h_\Delta^*(t)=\sum_{i \geq 0} h_i^*t^i$. 
Assume that $h_{k+1}^*=h_{k+2}^*=\cdots=h_{2k}^*=0$ for some $k \geq 3$. 
Then there exists a face $\Delta'$ of $\Delta$ such that $h^*_{\Delta'}(t)=\sum_{i=0}^k h_i^*t^i$. 
\end{thm}
As an immediate corollary, we obtain the following: 
\begin{cor}\label{maincor}
Let $\Delta$ be a lattice simplex with $h_\Delta^*(t)=\sum_{i \geq 0} h_i^*t^i$. 
Assume that $h_{k+1}^*=h_{k+2}^*=\cdots=h_{2k}^*=0$ for some $k \geq 3$. Then $(h_0^*,h_1^*,\ldots,h_k^*)$ satisfies degree $k$ condition. 
\end{cor}

A brief organization of this paper is as follows. 
First, in Section 2, we prepare a material which we will use in the proof of the main results. 
Next, in Section 3, we give a proof of Theorem \ref{mainthm}. 
Finally, in Section 4, we show that the condition $h_{k+1}^*=h_{k+2}^*=\cdots=h_{2k-1}^*=0$ of Theorem \ref{mainthm} 
is necessary for degree $k$ condition. See Proposition \ref{prop:main}. 


\bigskip

\section{Preliminary}

First of all, we prepare some tools for the computation of $h^*$-vectors of lattice simplices.

Let $\Delta \subset \R^d$ be a lattice simplex of dimension $d$ and let ${\bf v}_1,\ldots,{\bf v}_{d+1} \in \Z^d$ be the vertices of $\Delta$. 
We use the usual notation $[d+1]:=\{1,\ldots,d+1\}$. We define 
\begin{align*}
\Lambda_\Delta:=\left\{(r_1,\ldots,r_{d+1}) : 0 \leq r_j < 1 \text{ for }\forall j \in [d+1], \; 
\sum_{i=1}^{d+1} r_i{\bf v}_i \in \Z^d, \; \sum_{i=1}^{d+1} r_i \in \Z \right\}. 
\end{align*}
We see that $\Lambda_\Delta$ is a finite abelian group by its addition 
$$\alpha+\beta=(\{\alpha_1+\beta_1\},\ldots,\{\alpha_{d+1}+\beta_{d+1}\}) \in \Lambda_\Delta$$ 
for $\alpha,\beta \in \Lambda_\Delta$, where $\{r\}=r-\lfloor r \rfloor$ for $r \in \R$. 
Note that ${\bf 0}=(0, \ldots, 0) \in \Lambda_\Delta$ 
and $-\alpha=(\{1-\alpha_1\},\ldots,\{1-\alpha_{d+1}\}) \in \Lambda_\Delta$ for $\alpha \in \Lambda_\Delta$. 

We collect the notation concerning $\Lambda_\Delta$ which we will use. 
\begin{itemize}
\item For $\alpha \in \Lambda_\Delta$, let $\heit(\alpha)=\sum_{i=1}^{d+1}\alpha_i$. 
\item For $h=0,1,\ldots,d$, let $$\Lambda_\Delta^{(h)}=\{\alpha \in \Lambda_\Delta : \heit(\alpha)=h\}.$$ 
Note that $\Lambda_\Delta = \bigsqcup_{h=0}^d \Lambda_\Delta^{(h)}$. 
The $h^*$-polynomial of $\Delta$ can be computed as follows: 
\begin{align}\label{eq:h^*}h_\Delta^*(t)=\sum_{\alpha \in \Lambda_\Delta}t^{\heit(\alpha)}.\end{align} 
See \cite[Corollary 3.11]{BR15}. 
\item For $\alpha \in \Lambda_\Delta$, let $\supp(\alpha)=\{ i \in [d+1] : \alpha_i > 0\}$ 
and let $|\alpha|=|\supp(\alpha)|$. 
\item For $\Gamma \subset \Lambda_\Delta$, let $\supp(\Gamma)=\bigcup_{\alpha \in \Gamma}\supp(\alpha)$. 
\end{itemize}

We observe that for any $\alpha \in \Lambda_\Delta$, we have \begin{align}\label{obser}|\alpha|=\heit(\alpha)+\heit(-\alpha).\end{align} 
In fact, $$\heit(\alpha)+\heit(-\alpha)=\sum_{i \in \supp(\alpha)}\alpha_i+\sum_{i \in \supp(\alpha)} (1-\alpha_i)=|\alpha|.$$ 

\bigskip

\section{A proof of Theorem \ref{mainthm}}

Fix an integer $k$ with $k \geq 3$. 
In what follows, we consider a lattice simplex $\Delta$ with its $h^*$-polynomial $\sum_{i \geq 0}h_i^*t^i$ 
satisfying $h_{k+1}^*=\cdots=h_{2k}^*=0$. Namely, we have $\Lambda_\Delta^{(k+1)}=\cdots=\Lambda_\Delta^{(2k)}=\emptyset$. 
Our goal is to find a face $\Delta'$ of $\Delta$ such that $h_{\Delta'}^*(t)=\sum_{i=0}^k h_i^*t^i$.

\begin{lem}\label{lem1}
For any $\alpha \in \Lambda_\Delta^{(h)}$ and $0 \leq h \leq k$, we have $|\alpha| \leq k+h$. 
\end{lem}
\begin{proof}
Take $\alpha \in \Lambda_\Delta^{(h)}$ with $0 \leq h \leq k$. Then $\heit(\alpha) = h$. 

Suppose that $|\alpha| \geq k+h+1$. Then it follows from \eqref{obser} that 
$$\heit(-\alpha)=|\alpha|-\heit(\alpha) \geq k+1.$$ Hence $\heit(-\alpha) \geq 2k+1$ by our assumption. 
Let $m$ denote the order of $\alpha$, i.e., $m\alpha={\bf 0}$. Note that $-\alpha=(m-1)\alpha$. 
Take any $1 \leq j \leq m-1$. Then we see that 
$$\heit(j\alpha) = \sum_{i=1}^{d+1} \{ j\alpha_i \} \leq \sum_{i=1}^{d+1} \{ (j-1)\alpha_i \} + \sum_{i=1}^{d+1} \alpha_i 
=\heit((j-1)\alpha)+h.$$ 
On the one hand, $\heit(\alpha)=h \leq k$. On the other hand, $\heit((m-1)\alpha) \geq 2k+1$. 
This means together with $\heit(j\alpha) \leq \heit((j-1)\alpha)+k$ that there is $1 < j' <m-1$ with $k+1 \leq \heit(j'\alpha) \leq 2k$, 
a contradiction. 
\end{proof}

\begin{lem}\label{lem2}
Let $\Lambda'=\{\alpha \in \Lambda_\Delta : \heit(\alpha) \leq k\}$. Then the following statements hold: 
\begin{itemize}
\item[(a)] $\Lambda'$ is a subgroup of $\Lambda_\Delta$; 
\item[(b)] $|\supp(\Lambda')| \leq 4k-1$. 
\end{itemize}
\end{lem}
\begin{proof}
(a) In general, one sees that $\heit(\alpha+\beta) \leq \heit(\alpha)+\heit(\beta)$ for $\alpha,\beta \in \Lambda_\Delta$. 
Since $\heit(\alpha+\beta) \leq 2k$ for any $\alpha,\beta \in \Lambda'$, one has $\heit(\alpha+\beta) \leq k$. Thus, $\alpha+\beta \in \Lambda'$. 
Moreover, we see from Lemma \ref{lem1} that one has $\heit(-\alpha) = |\alpha|-\heit(\alpha) \leq k$ for $\alpha \in \Lambda'$. 
Hence, $-\alpha \in \Lambda'$. Clearly, ${\bf 0} \in \Lambda'$ since $\heit({\bf 0})=0$. Therefore, $\Lambda'$ is a subgroup of $\Lambda_\Delta$. 

\noindent
(b) Let $s=\max\{\heit(\alpha) : \alpha \in \Lambda'\}$. Then $s \leq k$. 
By \cite[Lemma 2.2 and Proposition 2.3 (a)]{Hig18}, we obtain that $|\supp(\Lambda')| \leq 2(2s)-1 \leq 4k-1$. 
\end{proof}


Now we are ready to give a proof of Theorem \ref{mainthm}. 

\begin{proof}[Proof of Theorem \ref{mainthm}]
Work with the same notation as in Lemma \ref{lem2}. Let $\Delta'=\conv(\{v_i : i \in \supp(\Lambda')\})$.  
Then $\Delta'$ is a face of $\Delta$. Note that $\supp(\Lambda_{\Delta'})=\supp(\Lambda') \subset [d+1]$ by the definition. 
In what follows, we prove that $h_{\Delta'}^*(t)=\sum_{i=0}^k h_i^*t^i$. Our goal is to show that $\Lambda_{\Delta'}=\Lambda'$. 
Clearly, we have $\Lambda' \subset \Lambda_{\Delta'}$, where we regard $\Lambda_{\Delta'}$ as a subgroup of $\Lambda_\Delta$. 
We prove that for any $\alpha \in \Lambda_\Delta \setminus \Lambda'$, we have $\alpha \not\in \Lambda_{\Delta'}$. 

Let $\alpha \in \Lambda_\Delta \setminus \Lambda'$. Then $\heit(\alpha) \geq 2k+1$. 
If $\heit(-\alpha) \leq k$, then $\heit(\alpha) \leq k$ by Lemma \ref{lem2} (a), a contradiction. Thus, $\heit(-\alpha) \geq 2k+1$. 
From \eqref{obser}, we see that 
\begin{align*}
2k+1 \leq \heit(-\alpha)=|\alpha|-\heit(\alpha) \leq |\alpha| - 2k-1, 
\end{align*}
i.e., $|\alpha| \geq 4k+2$. Now, Lemma \ref{lem2} (b) says that $|\gamma| \leq 4k-1$ for any $\gamma \in \Lambda'$. 
Since $|\alpha| \geq 4k+2>4k-1$, one sees that $\supp(\alpha) \not\subset \supp(\Lambda')$. 
Since $\supp(\Lambda') = \supp(\Lambda_{\Delta'})$, we conclude that $\alpha \not\in \Lambda_{\Delta'}$, as required. 
\end{proof}


\bigskip

\section{Examples showing that $h_{k+1}^*=h_{k+2}^*=\cdots=h_{2k-1}^*=0$ is necessary}

We will provide the example which shows that Theorem \ref{mainthm} is not true 
if we drop the assumption $h_j^*=0$ for each $k+1 \leq j \leq 2k-1$ (Proposition \ref{prop:main}). 

We recall a useful construction of lattice polytopes. Let $P \subset \R^d$ and $Q \subset \R^e$ be lattice polytopes. We define the \textit{join} of $P$ and $Q$: 
$$P \star Q:=\conv(\{(0,{\bf x}, {\bf 0}_e) : {\bf x} \in P\} \cup \{(1, {\bf 0}_d, {\bf y}) : {\bf y} \in Q\}) \subset \R^{d+e+1},$$ 
where ${\bf 0}_d$ (resp. ${\bf 0}_e$) denotes the origin of $\R^d$ (resp. $\R^e$). 
Note that $P \star Q$ is a simplex if and only if so are $P$ and $Q$. 
It is known from \cite[Lemma 1.3]{HT09} that $$h_{P \star Q}^*(t)=h_P^*(t)h_Q^*(t).$$ 

\begin{lem}\label{construct}
For any positive integers $a,b,k,\ell$, there exists a lattice simplex $\Delta$ 
with its $h^*$-polynomial $1+at^k+bt^\ell+(a+b)t^{k+\ell}$. 
\end{lem}
\begin{proof}
For positive integers $c$ and $m$, let $\Delta_{c,m}$ be a lattice simplex with its $h^*$-polynomial $1+ct^m$. 
Such a simplex is given in e.g., \cite[Example 3.22]{BR15}. 
A required simplex can be obtained by $\Delta_{a,k} \star \Delta_{b,\ell}$. 
\end{proof}

\begin{lem}\label{HHH}
For any prime number $p$ with $p \geq 5$ and integers $i,j$ with $2 \leq i < j$, 
there exists no lattice polytope whose $h^*$-polynomial is $1+t^i+(p-2)t^j$. 
\end{lem}
\begin{proof}
Suppose that there is a lattice polytope $P$ with $h^*_P(t)=1+t^i+(p-2)t^j$. 
Since the linear term of $h_P^*(t)$ vanishes, $P$ must be a simplex. 
Now, \cite[Theorem 1.1 (a)]{Hig14} says that the $h^*$-polynomial with a prime normalized volume should satisfy 
$h_{i+1}^*=h_{s-i}^*$ for any $i=0,1,\ldots,s$, where $s$ is the degree of lattice polytope. 
However, since $p \geq 5$, this never happens, a contradiction. 
\end{proof}

\begin{prop}\label{prop:main}
For any $k \geq 3$ and $k+1 \leq j \leq 2k-1$, there exists a lattice simplex $\Delta$ with $h^*_\Delta(t)=\sum_{i \geq 0}h_i^*t^i$ 
such that \begin{itemize}
\item $h_{k+1}^*=\cdots=h_{j-1}^*=h_{j+1}^*=\cdots=h_{2k-1}^*=0$, 
\item $h_j^* \neq 0$ and 
\item $(h_0^*,\ldots,h_k^*)$ does not satisfy degree $k$ condition. 
\end{itemize}
\end{prop}
\begin{proof}
When $j \geq k+2$, let $a=j-k$ and $b=k$. Note that $2 \leq a \leq k-1$. 
Then Lemma \ref{construct} says that there exists a lattice simplex $\Delta$ with its $h^*$-polynomial 
$1+t^a+(p-2)t^b+(p-1)t^{a+b}=1+t^{j-k}+(p-2)t^k+(p-1)t^j$, where we let $p$ be a prime number with $p \geq 5$. 
Therefore, by Lemma \ref{HHH}, this $h^*$-polynomial enjoys the required properties. 

In the case $j=k+1$ and $k \geq 4$, there exists a lattice simplex $\Delta$ with its $h^*$-polynomial 
$1+t^2+(p-2)t^{k-1}+(p-1)t^{k+1}$ by Lemma \ref{construct}, where we let $p$ be a prime number with $p \geq 5$, 
and this $h^*$-polynomial enjoys the required properties by Lemma \ref{HHH}. 

In the case $j=k+1$ and $k=3$, consider the lattice simplex 
$$\mathrm{conv}({\bf 0},{\bf e}_1,\ldots,{\bf e}_4,{\bf e}_1+4{\bf e}_2+7{\bf e}_3+8{\bf e}_4+9{\bf e}_5\}) \subset \R^5.$$ 
Then the $h^*$-polynomial of this simplex is equal to $1+2t^2+4t^3+2t^4$. 
Since $(1,0,2,4)$ does not satisfy degree $3$ condition by Lemma \ref{HHH}, we conclude that this $h^*$-polynomial enjoys the required properties. 
\end{proof}

Proposition \ref{prop:main} says that we cannot drop the assumption $h_j^*=0$ in Corollary \ref{maincor} (Theorem \ref{mainthm}, too) for any $k+1 \leq j \leq 2k-1$. 
\begin{rem}\label{chuui}
There exists a lattice simplex $\Delta$ of dimension $3k-1$ such that $h_\Delta^*(t)=1+t^k+t^{2k}$ for any $k \geq 2$. 
In fact, we may set $\Delta=\mathrm{conv}(\{{\bf 0}, \eb_1,\ldots,\eb_{d-1},2(\eb_1+\cdots+\eb_{d-1})+3\eb_d)\})$, 
where $d=3k-1$ and $\eb_1,\ldots,\eb_d$ denote the unit vectors of $\R^d$. 
Moreover, we see from \cite[Section 4.2]{HHL12} that this is the only lattice simplex of dimension $3k-1$ up to unimodular equivalence 
whose $h^*$-polynomial is $1+t^k+t^{2k}$. On the other hand, we see that this $h^*$-vector is {\bf shifted symmetric} (see \cite{Hig10}), 
i.e., $h_{i+1}^*=h_{d-i}^*$ for each $0 \leq i \leq d-1$. 
Thus, the $h^*$-polynomial of every face of $\Delta$ is $1$ by \cite[Theorem 2.1]{Hig10}. 

Therefore, Theorem \ref{mainthm} is not true under the assumption $h_{k+1}^*=\cdots=h_{2k-1}^*=0$, i.e., the assumption $h_{2k}^*=0$ is necessary. 

However, Corollary \ref{maincor} might be still true even if $h_{2k}^* \neq 0$. In fact, $1+t^k$ is a possible $h^*$-polynomial of a lattice simplex $\Delta_{1,k}$. 
\end{rem}

\bigskip

\section{Future Questions}

Finally, we suggest two questions: 
\begin{quest}\label{q1}
Let $P$ be a lattice {\bf polytope} with $h_P^*(t)=\sum_{i \geq 0}h_i^*t^i$. 
Assume that $h_{k+1}^*=h_{k+2}^*=\cdots=h_{2k}^*=0$. Then does $(h_0^*,h_1^*,\ldots,h_k^*)$ satisfy degree $k$ condition? 
\end{quest}

\begin{quest}\label{q2}
Let $\Delta$ be a lattice simplex with $h_\Delta^*(t)=\sum_{i \geq 0}h_i^*t^i$. 
Assume that $h_{k+1}^*=h_{k+2}^*=\cdots=h_{2k-1}^*=0$. Then does $(h_0^*,h_1^*,\ldots,h_k^*)$ satisfy degree $k$ condition? 
Is this also true for polytopes?
\end{quest}


\bigskip




\begin{thebibliography}{10}
\bibitem{Balletti} G. Balletti, Enumeration of lattice polytope by their volume, arXiv:1811.03357. 
\bibitem{BH18} G. Balletti and A. Higashitani, Universal inequalities in Ehrhart Theory, \textit{Israel J. Math.}, {\bf 227} 843--859, (2018). 
\bibitem{BR15} M. Beck and S. Robins, ``Computing the continuous discretely'', Undergraduate Texts in Mathematics. Springer, New York, second edition, (2015). 
\bibitem{Ehr62} E. Ehrhart. Sur les poly\`edres rationnels homoth\'etiques \`a {$n$}\ dimensions, \textit{C. R. Acad. Sci. Paris}, {\bf 254} 616--618, (1962). 
\bibitem{HT09} M. Henk and M. Tagami, Lower bounds on the coefficients of {E}hrhart polynomials, \textit{European J. Combin.}, {\bf 30} 70--83, (2009). 
\bibitem{Hib94} T. Hibi, A lower bound theorem for Ehrhart polynomials of convex polytopes, \textit{Adv. Math.}, {\bf 105} 162--165, (1994).
\bibitem{HHL12} T. Hibi, A. Higashitani and N. Li, Hermite normal forms and $\delta$-vector, \textit{J. Comb. Theory Ser. A} {\bf 119} 1158--1173,  (2012). 
\bibitem{Hig10} A. Higashitani, Shifted symmetric $\delta$-vectors of convex polytopes, \textit{Discrete Math.} {\bf 310} 2925--2934,  (2010). 
\bibitem{Hig14} A. Higashitani, Ehrhart polynomials of integral simplices with prime volumes, \textit{INTEGERS}, {\bf 14} 1--15, (2014). 
\bibitem{Hig18} A. Higashitani, Lattice simplices of maximal dimension with a given degree, \textit{Michigan Math. J.}, to appear. 
\bibitem{HKN16} J. Hofscheier, L. Katth\"an and B. Nill, Ehrhart theory of spanning lattice polytopes, \textit{Int. Math. Res. Not.}, to appear. 
\bibitem{Sco76} P. R. Scott, On convex lattice polygons, \textit{Bull. Austral. Math. Soc.}, {\bf 15} 395--399, (1976). 
\bibitem{Sta80} R. P. Stanley, Decompositions of rational convex polytopes, \textit{Ann. Discrete Math.}, {\bf 6} 333--342, (1980). 
\bibitem{Tre10} J. Treutlein, Lattice polytopes of degree $2$, \textit{J. Combin. Theory Ser. A}, {\bf 117} 354--360, (2010). 
%
%
%
%
%
%
%
%
%
%
%
%
%
\end{thebibliography}
\end{document}